\documentclass[11pt,a4paper]{article}

\usepackage{hyperref}
\usepackage{tikz}
\usepackage{amsmath,amsfonts,amssymb,latexsym,graphics,epsfig,url}
\usepackage{color}
\usepackage{amsthm,enumerate}
\usepackage[english]{babel}
\usepackage{graphicx}
\usepackage{soul}

\usepackage{lipsum}

\newcommand\blfootnote[1]{%
  \begingroup
  \renewcommand\thefootnote{}\footnote{#1}%
  \addtocounter{footnote}{-1}%
  \endgroup
}
\newtheorem{theo}{Theorem}[section]
\newtheorem{proposition}[theo]{Proposition}
\newtheorem{lemma}[theo]{Lemma}

\newtheorem{corollary}[theo]{Corollary}

\newtheorem{example}{Example}

\input amssym.def
\newsymbol\rtimes 226F
\newfont{\nset}{msbm10}

\def\Par{\pi}

\def\G{\Gamma}
\def\Re{\mathbb R}

\def\A{{\mbox {\boldmath $A$}}}
\def\B{{\mbox {\boldmath $B$}}}
\def\C{{\mbox {\boldmath $C$}}}

\def\D{{\mbox {\boldmath $D$}}}
\def\E{{\mbox {\boldmath $E$}}}
\def\G{\Gamma}

\def\I{{\mbox {\boldmath $I$}}}
\def\J{{\mbox {\boldmath $J$}}}

\def\Ei{{\cal E}}

\def\Q{\mbox{\boldmath $Q$}}

\def\U{\mbox{\boldmath $U$}}
\def\V{\mbox{\boldmath $V$}}

\def\A{{\mbox {\boldmath $A$}}}

\def\matrix0{{\mbox {\boldmath $O$}}}

\def\S{\mbox{\boldmath $S$}}

\def\e{{\mbox{\boldmath $e$}}}

\def\j{{\mbox{\boldmath $j$}}}

\def\u{{\mbox{\boldmath $u$}}}
\def\v{{\mbox{\boldmath $v$}}}

\def\vec0{\mbox{\bf 0}}

\def\diag{\mathop{\rm diag }\nolimits}
\def\dist{\mathop{\rm dist }\nolimits}

\def\ecc{\mathop{\rm ecc }\nolimits}

\def\Ker{\mathop{\rm Ker }\nolimits}

\def\ev{\mathop{\rm ev }\nolimits}
\def\spec{\mathop{\rm sp }\nolimits}

\def\G{\Gamma}

\def\Re{\mathbb R}

\title{A general method to obtain the spectrum and \\ local spectra of a graph from its regular partitions
\thanks{This research is partially supported by the project 2017SGR1087 of the Agency for the Management of University and Research Grants (AGAUR) of the Government of Catalonia.
}}

\author{C. Dalf\'o\\
\small{Departament  de Matem\`atica}\\
\small{Universitat de Lleida, Igualada (Barcelona), Catalonia}\\
\vspace{.5cm}
\small{\url{cristina.dalfo@matematica.udl.cat}}\\
M. A. Fiol\\
\small{Deptartament de Matem\`atiques}\\
\small{Barcelona Graduate School of Mathematics,}\\
\small{Universitat Polit\`ecnica de Catalunya, Barcelona, Catalonia}\\
\small{\url{miguel.angel.fiol@upc.edu}}
}

\begin{document}
	
\maketitle



\begin{abstract}
It is well known that, in general, part of the spectrum of a graph can be obtained from the adjacency matrix of its quotient graph given by a regular partition.
In this paper, we propose a method to obtain all the spectrum, and also the local spectra, of a graph $\G$ from the quotient matrices of some of its regular partitions.
As examples, it is shown how to find the eigenvalues and (local) multiplicities of walk-regular, distance-regular, and distance-biregular graphs.
\end{abstract}

\blfootnote{\begin{minipage}[l]{0.3\textwidth} \includegraphics[trim=10cm 6cm 10cm 5cm,clip,scale=0.15]{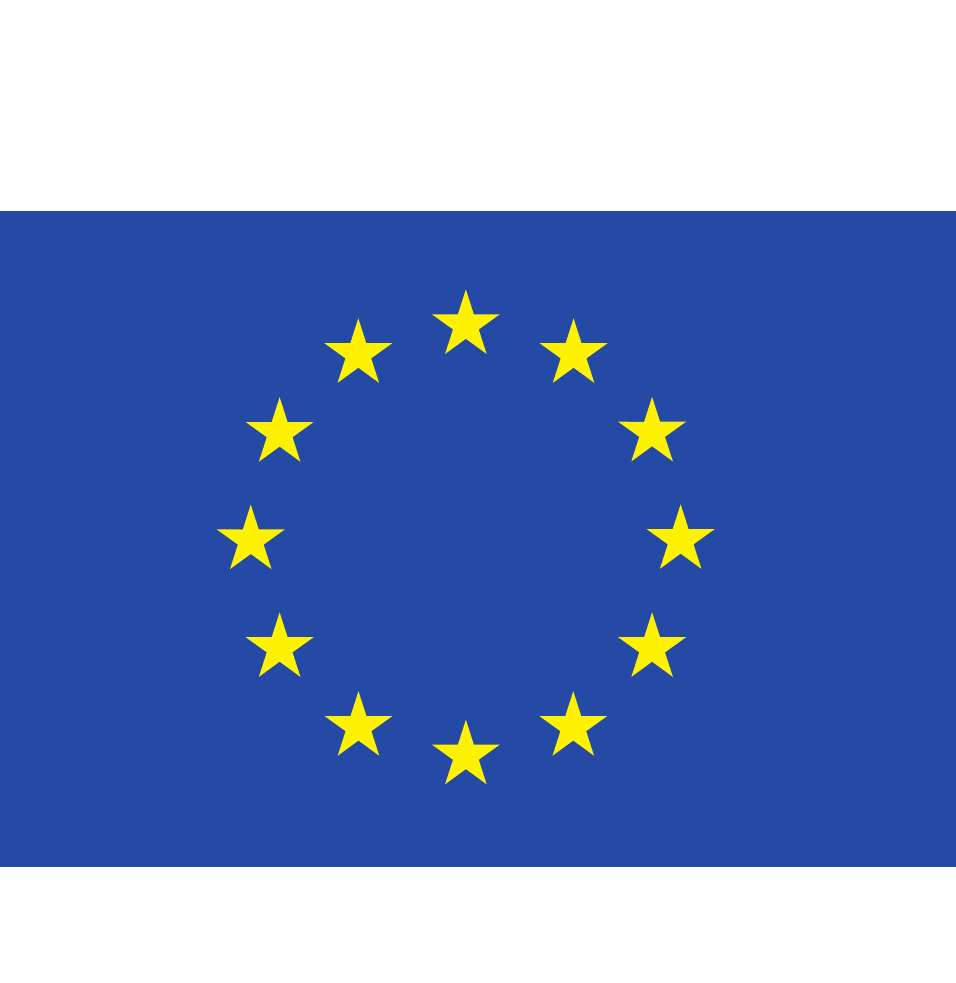} \end{minipage}  \hspace{-2cm} \begin{minipage}[l][1cm]{0.79\textwidth}
	The first author has also received funding from the European Union's Horizon 2020 research and innovation programme under the Marie Sk\l{}odowska-Curie grant agreement No 734922.
\end{minipage}}


\section{Preliminaries}
 First, let us recall some basic concepts and define our generic notation for graphs.

\subsection{Some notions on graphs and their spectra}

Throughout this paper, $\G=(V,E)$ denotes a simple and connected 
graph with order $n=|V|$, size $m=|E|$, and adjacency matrix
$\A$. The {\it distance} between two vertices $u$ and $v$ is
denoted by $\dist (u,v)$, so that the {\it eccentricity} of
a vertex $u$ is $\ecc(u)=\max_{v\in V}\dist (u,v)$, and the {\it
diameter} of the graph is $D=\max_{u\in V}\ecc(u)$. The set of
vertices at distance $i$, from a given vertex $u\in V$ is
denoted by $\G_i(u)$, for $i=0,1,\dots,D$, and we write $\G(u)=\G_1(u)$ for short. The degree of a
vertex $u$ is denoted by $\delta(u)=|\G_1(u)|$. The {\em
distance-$i$ graph} $\G_i$ is the graph with vertex set $V$, and
where two vertices $u$ and $v$ are adjacent if and only if
$\dist(u,v)=i$ in $\G$. Its adjacency matrix $\A_i$ is usually
referred to as the {\em distance-$i$ matrix} of $\G$.
The spectrum of a graph $\G$ of its adjacency matrix $\A(=\A_1)$  is denoted by
$$
\spec \G = \spec \A = \{\theta_0^{m(\theta_0)},\theta_1^{m(\theta_1)},\dots,
\theta_d^{m(\theta_d)}\},
$$
where the  different eigenvalues of $\G$, whose set is denoted by $\ev \G$, are in decreasing order,
$\theta_0>\theta_1>\cdots >\theta_d$, and the superscripts
stand for their multiplicities $m(\theta_i)$ for $i=0,\ldots,d$. In
particular, note that $m(\theta_0)=1$, since $\G$ is
connected, and $m(\theta_0)+m(\theta_1)+\cdots+m(\theta_d)=n$.
Alternatively, if we include repetitions, the eigenvalues of $\G$ are denoted as
$\lambda_1\ge \lambda_2\ge \cdots \ge \lambda_n.$


\subsection{Projections and local spectra}
For any graph with eigenvalue $\theta_i$ having multiplicity $m(\theta_i)$, its corresponding {\it $($principal\/$)$ idempotent} 
can be computed as $\E_i=\V_i\V_i^\top$, where $\V_i$ is the $n\times m(\theta_i)$ matrix whose columns form an orthonormal basis of the eigenspace $\Ei_i=\Ker (\A-\theta_i \I)$.
For instance, when $\G$ is a $\delta$-regular graph on $n$ vertices,
its largest eigenvalue $\theta_0=\delta$ has eigenvector $\j$,
the all-$1$ (column) vector, and corresponding idempotent
$\E_0=\frac{1}{n}\j\j^{\top}=\frac 1n \J$, where $\J$ is the all-$1$ matrix.

Alternatively, for every $i=0,1,\dots,d$, the orthogonal projection of $\Re^n$
onto the eigenspace $\Ei_i$ is given by
the Lagrange interpolating polynomial
\begin{equation}
\label{Li}
L_i=\frac{1}{\phi_i}\prod_{\stackrel{j=0}{j\neq i}}^d
(x-\theta_j) 
\end{equation}
of degree $d$, where $\phi_i=\prod_{j=0,j\neq i}^d
(\theta_i-\theta_j)$.
These polynomials
satisfy $L_i(\theta_i)=1$ and $L_i(\theta_j)=0$ for $j\neq i$. The idempotents are, then, 
$$
\E_i=L_i(\A)=\frac{1}{\phi_i}\prod_{\stackrel{j=0}{j\neq i}}^d
(\A-\theta_j\I),
$$
and they are known to satisfy the following properties (see, for instance, Godsil~\cite[p. 28]{g93}):
\begin{itemize}
\item[$(a)$]
$\E_i\E_j=\delta_{ij}\E_i$;
\item[$(b)$]
$\A\E_i=\theta_i\E_i$;
\item[$(c)$]
$p(\A)=\displaystyle\sum_{i=0}^d p(\theta_i)\E_i$, for any polynomial $p\in
\Re[x]$.
\end{itemize}
In particular, when $p(x)=x$ in $(c)$, we have the so-called {\em spectral decomposition theorem}: $\A=\sum_{i=0}^d \theta_i\E_i$.
The {\em $(u$-$)$local multiplicities} of the eigenvalue $\theta_i$, introduced by Fiol and Garriga in \cite{fg2}, were
defined as the square norm of the projection of $\e_u$ onto the eigenspace $\Ei_i$, where $\e_u$ is the unitary characteristic vector of a vertex $u\in V$. That is,
$$
m_u(\theta_i) = \|\E_i\e_u\|^2
= \langle\E_i\e_u,\e_u\rangle =
(\E_i)_{uu},\qquad u\in V,\ i =0,1,\dots,d.
$$
Notice that, in fact, $m_u(\theta_i)=\cos^2\beta_{ui}$, where $\beta_{ui}$ is the angle between $\e_u$ and $\E_i \e_u$. The values $\cos \beta_{ui}$,  for $u\in V$ and $i=0,\ldots,d$, were formally introduced by Cvetkovi\'c as the `angles' of $\G$ (see, for instance, Cvetkovi\'c and Doob~\cite{cd85}).

The local multiplicities can be seen as a generalization of the (standard) multiplicities when the graph is `seen' from the `base vertex' $u$. Indeed, they satisfy the following properties (see Fiol and Garriga~\cite{fg2}):
\begin{align}
\sum_{i=0}^d m_u(\theta_i) &=  1; \label{sum-locmul1}\\
\sum_{u\in V} m_u(\theta_i)&=  m(\theta_i),\qquad  i=0,1,\dots,d. \label{sum-locmul2}
\end{align}

If $\mu_0(=\theta_0)>\mu_1>\cdots>\mu_{d_u}$ represent the eigenvalues of $\G$ with non-null $u$-local multiplicity, we define the {$u$-local spectrum} of $\G$ as
$$
\spec_u\G=\{\mu_0^{m_u(\mu_0)},\mu_1^{m_u(\mu_1)},\dots,
\mu_{d_u}^{m_u(\mu_{d_u})}\}.
$$

By analogy with the local multiplicities, which correspond to
the diagonal entries of the idempotents, Fiol, Garriga, and
Yebra~\cite{fgy99} defined the {\it crossed $(uv$-$)$local
multiplicities\/} of the eigenvalue $\theta_i$,  denoted by
$m_{uv}(\theta_i)$, as
$$
m_{uv}(\theta_i)=\langle\E_i\e_u,\E_i\e_v\rangle
=\langle\E_i\e_u,\e_v\rangle=(\E_i)_{uv}, \qquad u,v\in V,\ i =0,1,\dots,d.
$$
(Thus, in particular, $m_{uu}(\theta_i)=m_{u}(\theta_i)$.)
These parameters allow us to compute the number of walks of
length $\ell$ between two vertices $u,v$ in the following way:
\begin{equation}
\label{crossed-mul->num-walks} a_{uv}^{({\ell})} =(\A^{\ell})_{uv}=
\sum_{i=0}^d m_{uv}(\theta_i)\theta_i^{\ell}, \qquad \ell =0,1,\dots
\end{equation}
Conversely, the values  $a_{uv}^{(\ell)}$, for $\ell=0,1,\ldots,d$,  determine 
 the crossed local multiplicities
 $m_{uv}(\theta_i)$. (Indeed, notice that the coefficients  of the system in \eqref{crossed-mul->num-walks} are the entries of a Vandermonde matrix).

\section{Regular partitions and local spectra}
\label{sec:reg-part}

Let $\G=(V,E)$ be a graph with adjacency matrix $\A$. A partition $\Par=(V_1,\ldots,
V_m)$ of its vertex set $V$ is called {\em regular} (or {\em equitable})
whenever, for any $i,j=1,\ldots,m$, the {\em intersection numbers}
$b_{ij}(u)=|\G(u)\cap V_j|$, where $u\in V_i$, do not depend on the vertex $u$ but only on the subsets (usually called {\em classes} or {\em
cells}) $V_i$ and $V_j$. In this case, such numbers are simply written as $b_{ij}$,
and the $m\times m$ matrix $\B=(b_{ij})$ is referred to as the {\em quotient matrix} of $\A$ with respect to $\Par$. This is also represented by the {\em quotient (weighted) graph} $\pi(\G)$ (associated to the partition $\pi$),
with vertices representing the cells, and there is an edge with weight $b_{ij}$ between  vertex
$V_i$ and vertex $V_j$ if and only if $b_{ij}\neq 0$. Of course, if $b_{ii}>0$, for some $i=1,\ldots,m$, the quotient graph $\pi(\G)$ has loops.

The {\em characteristic matrix} of (any) partition $\Par$ is the $n\times m$ matrix
$\S=(s_{ui})$ whose $i$-th column is the characteristic vector of $V_i$, that is, $s_{ui}=1$ if $u\in V_i$, and $s_{ui}=0$ otherwise. In terms of this matrix,
we have the following characterization of regular partitions (see Godsil \cite{g93}).

\begin{lemma}[\cite{g93}]
Let $\G=(V,E)$ be a graph with adjacency matrix $\A$, and vertex partition $\Par$
with characteristic matrix $\S$. Then, $\Par$ is regular if and only if there
exists an $m\times m$ matrix $\C$ such that $\S\C=\A\S$. Moreover, $\C=\B$, the
quotient matrix of $\A$ with respect to $\Par$.
\end{lemma}


Using the above lemma, it can be proved that  $\spec \B\subseteq \spec \A$. Moreover, we have the following result by the authors~\cite{df17}.

\begin{lemma}[\cite{df17}]
\label{lemma:CrisMA}
Let $\G$ be a graph with adjacency matrix $\A$. Let $\pi=(V_1,\ldots,V_m)$ be a regular partition of $\G$, with quotient matrix $\B$. Then, the number of $\ell$-walks from any vertex $u\in V_i$ to all vertices of $V_j$ is the
$ij$-entry $b_{ij}^{(\ell)}$ of $\B^{\ell}$.
\end{lemma}

By using the last lemma, now we have the next result.
\begin{lemma}
\label{lemma-walks}
Let $\G$ be a graph with adjacency matrix $\A$, and $\pi=(V_1,\ldots,V_m)$ a regular partition of $\G$ with quotient matrix $\B$. If $V_1=\{u\}$, then the number of $\ell$-walks from vertex $u$ to a vertex $v\in V_j$, for $j=1,\ldots,m$, only depends on $j$:
\begin{equation}
(\A^{\ell})_{uv}=a_j^{(\ell)}=\frac{1}{|V_j|}(\B^{\ell})_{1j}.
\end{equation}
\end{lemma}

\begin{proof}
To prove that the number of $\ell$-walks between $u$ and $v\in V_j$ is a constant, we use induction on $\ell$. The result is clearly true for $\ell=0$, since $\B^0=\I$, and for
$\ell=1$ because of the definition of $\B$. Suppose that the result holds for some $\ell>1$. Then, the set of walks of length $\ell+1$ from $u$ to $v\in V_j$ is obtained from the set of $\ell$-walks from $u$ to vertices $w\in V_h$ adjacent to $v$. Then, the number of these 
walks is
$$
(\A^{\ell+1})_{uv}=\sum_{h=1}^m\sum_{w\in \G(v)\cap V_h}(\A^{\ell})_{uw}=\sum_{h=1}^m b_{jh}a_h^{(\ell)}=a^{(\ell+1)}_{j},
$$
 as claimed.
\end{proof}


Let $\B$ be a quotient (diagonalizable) $m\times m$ matrix as above, with $\spec \B=\{\tau_0^{m(\tau_0)},\tau_1^{m(\tau_1)},$ $\ldots,\tau_e^{m(\tau_e)}\}$, and let $\D=\diag(\tau_0,\tau_1,\ldots,\tau_e)$.
Let $\Q$ be the $m\times m$ matrix that diagonalizes $\B$, that is, $\Q^{-1}\B\Q=\D$. For $i=0,\ldots,e$, let $\V_i$ be the $m\times m(\tau_i)$ matrix formed by the columns of $\Q$ corresponding to the right $\tau_i$-eigenvectors of $\B$. Let $\U_i$ be the $m(\tau_i)\times m$ matrix formed by the corresponding rows of $\Q^{-1}$, which are the left $\tau_i$-eigenvectors of $\B$. Then, the $i$-th idempotent of $\B$ is $\overline{\E}_i=\V_i\U_i$. Moreover, $\overline{\E}_i$ can be computed as in the case of the (symmetric) adjacency matrix
by using the Lagrange interpolating polynomial $\overline{L}_i$ satisfying $\overline{L}_i(\tau_j)=\delta_{ij}$:
$$
\overline{\E}_i=\overline{L}_i(\B)=
\frac{1}{\displaystyle\prod_{\stackrel{j=0}{j\neq i}}^e
(\tau_i-\tau_j)}\prod_{\stackrel{j=0}{j\neq i}}^e
(\B-\tau_j\I).
$$

The above results yield a simple method to compute the local spectra of a vertex $u$ in a given regular partition or, more generally, the crossed multiplicities between $u$ and any other vertex $v$. Besides, with the union of the local spectra (applying \eqref{sum-locmul2}) of the different classes of vertices according to their corresponding regular partitions (that is, we `hung' the quotient graph from every one of the different classes of vertices), we obtain all the spectrum of the original graph.
Thus, the main result is the following.

\begin{theo}
\label{main-theo}
Let $\G$ be a graph with adjacency matrix $\A$ and set of different eigenvalues $\ev \G=\{\theta_0,\theta_1,\ldots,\theta_d\}$. Let $\pi=(V_1,\ldots,V_m)$ be a regular partition of $\G$, with $V_1=\{u\}$. Let $\B$ be the quotient matrix of $\pi$, with set of different eigenvalues $\ev \B=\{\tau_0,\tau_1,\ldots,\tau_e\}\subseteq \ev \G$. Let $L_i$ and $\overline{L}_i$ be the Lagrange interpolating polynomials satisfying $L_i(\theta_j)=\delta_{ij}$ for $i,j=0,\ldots,d$, and  $\overline{L}_i(\tau_j)=\delta_{ij}$ for $i,j=0,\ldots,e$, respectively. Let $\E_i=L_i(\A)$ and  $\overline{\E}_i=\overline{L}_i(\B)$ be the corresponding idempotents. Then, for every vertex $v\in V_j$, the crossed $uv$-local multiplicity of $\theta_i$ is
\begin{align}
m_{uv}(\theta_i) &= \frac{1}{|V_j|}(L_i(\B))_{1j},\qquad i=0,1,\ldots,d,
\label{eq1xmul}
\end{align}
or, alternatively,
\begin{align}
\!\!\!\!\!\!\!\!\!\!\!\!\!\!
m_{uv}(\theta_i) &=
\left\{
\begin{array}{cc}
\frac{1}{|V_j|}(\overline{\E}_i)_{1j}&  \mbox{if $\theta_i\in \ev\B$},\\
0 & otherwise.
\end{array}
\right. \label{eq2xmul}
\end{align}
\end{theo}

\begin{proof}
Let $L_i(x)=\sum_{r=0}^d\zeta_r x^r$. Then, for every $v\in V_j$ and $i=0,1,\ldots,d$, and using Lemma \ref{lemma-walks}, we have
\begin{align*}
 m_{uv}(\theta_i) &= (\E_i)_{uv}=(L_i(\A))_{uv}
   =  \sum_{r=0}^d\zeta_r (\A^r)_{uv}\\
    & = \frac{1}{|V_j|}\sum_{r=0}^d\zeta_r (\B^r)_{1j}
   =\frac{1}{|V_j|}(L_i(\B))_{1j},
 \end{align*}
which proves \eqref{eq1xmul}. To prove \eqref{eq2xmul}, note first that, by the spectral decomposition theorem,
$\B^r=\sum_{i=0}^e \tau_i^r\overline{\E}_i$. Then, by Lemma \ref{lemma-walks}, the numbers of $\ell$-walks from $u$ to $v\in V_j$ are
$$
a_{uv}^{(\ell)}=
\frac{1}{|V_j|}\sum_{i=0}^e\tau_i^{\ell} (\E_i)_{1j},\qquad \ell=0,\ldots,d,
$$
which, as already commented, determine 
the local multiplicities because of the system of equations
$$
\sum_{i=0}^d \theta_i^{\ell} m_{uv}(\theta_i)=a_{uv}^{(\ell)}, \qquad \ell=0,\ldots,d.
$$
But a (the) solution of this system is obtained when the multiplicities $m_{uv}(\theta_i)$, for $i=0,\ldots,d$, are given by \eqref{eq2xmul}, as claimed.
\end{proof}

In particular, notice that this result allows us to compute the {$u$-local spectrum} of $\G$ as
\begin{equation}
\spec_u\G=\{\tau_0^{m_u(\tau_0)},\tau_1^{m_u(\tau_1)},\dots,
\tau_{e}^{m_u(\tau_{e})}\}, 
\label{local-sp}
\end{equation}
where $\tau_i\in\ev \B$ and $m_u(\tau_i)=(\overline{\E}_i)_{11}$, for $i=0,\ldots,e$.

Let us show an example.
\begin{example}
Let $\Delta=\G+u$ be the cone of a $k$-regular graph on $n$ vertices, where the `new' vertex $u$ is joined to all vertices of $\G$. Then, $\Delta$ has a regular partition with quotient matrix
$$
\B=\left(
\begin{array}{cc}
0 & n \\
1 & k
\end{array}
\right)
$$
and eigenvalues $\theta_0=\frac{1}{2}(k+\sqrt{k^2+4n})$ and $\theta_1=\frac{1}{2}(k-\sqrt{k^2+4n})$. (Notice that the first expression can be rewritten as $k=\theta_0-\frac{\theta_0}{n}$, in agreement with the results of Dalf\'o, Fiol, and Garriga \cite{dfg11}.) Thus, the idempotents of $\B$ turn out to be
$\overline{\E}_0=(\B-\theta_1\I)/(\theta_0-\theta_1)$ and $\overline{\E}_1=(\B-\theta_0\I)/(\theta_1-\theta_0)$.
Consequently, Theorem \ref{main-theo} implies that the local $u$-spectrum of $\Delta$ is
$$
\spec_u \Delta=\left\{\frac{1}{2}\left(k+\sqrt{k^2+4n}\right)^{m_u(\theta_0)}, \ \frac{1}{2}\left(k-\sqrt{k^2+4n}\right)^{m_u(\theta_1)}\right\},
$$
where $m_u(\theta_0)=\frac{1}{2}(1-k/\sqrt{k^2+4n})$ and
$m_u(\theta_1)=\frac{1}{2}(1+k/\sqrt{k^2+4n})$.
\end{example}

Another simple consequence of our main result is obtained for simple eigenvalues of $\B$.
\begin{corollary}
Let $\G$ be a graph with adjacency matrix $\A$ and a set of different eigenvalues $\ev \G=\{\theta_0,\theta_1,\ldots,\theta_d\}$, $\pi=(V_1,\ldots,V_m)$ a regular partition of $\G$ with $V_1=\{u\}$, $\B$ the quotient matrix of $\pi$ with a set of different eigenvalues $\ev \B=\{\tau_0,\tau_1,\ldots,\tau_e\}\subseteq \ev \G$, and $\E_i$ and $\overline{\E}_i$ the corresponding idempotents. Suppose that, for some $i$, $\theta_i\in \ev\A\cap\ev\B$ has multiplicity $1$. Let $\u_i=(u_{i1},\ldots,u_{im})$ and $\v_i=(v_{1i},\ldots,v_{im})^{\top}$ be the left and right eigenvectors of $\B$, respectively, corresponding to the eigenvalue $\theta_i$.
Then, for every vertex $v\in V_j$, the crossed $uv$-local multiplicity of $\theta_i$ in $\G$ is
\begin{equation}
m_{uv}(\tau_i)=\frac{1}{|V_j|}\frac{v_{1i}u_{ij}}{\langle\u_i,\v_i \rangle}, \qquad j=1,\ldots,m. \label{Biggs-general}
\end{equation}
\end{corollary}

\begin{proof}
Let $\Q$ be a matrix that diagonalizes $\B$.
If, for some constants $\alpha$ and $\beta$, we have that $\alpha\u_i$ and $\beta \v_i$ are the corresponding row of $\Q^{-1}$ and column of $\Q$, respectively, then $(\Q^{-1}\Q)_{ii}=1$ implies that $\alpha\beta=\langle \u_i,\v_i\rangle^{-1}$. Thus, $(\overline{\E}_i)_{1j}=(\alpha\v_i\cdot \beta\u_i)_{1j}=\frac{v_{1i}u_{ij}}{\langle \u_i,\v_i\rangle}$ (where `$\cdot$' stands for the matrix product), and the result follows from \eqref{eq2xmul}.
\end{proof}

Alternatively, if $\u_i$ and $\v_i$ are already taken from the corresponding row and column of $\Q^{-1}$ and $\Q$, respectively, then $\alpha=\beta=1$, and \eqref{Biggs-general} can be simply written as
\begin{equation}
m_{uv}(\tau_i)=\frac{1}{|V_j|}(\Q^{-1})_{ij}(\Q)_{1i}, \qquad j=1,\ldots,m.
\label{Biggs-general2}
\end{equation}

\section{The local spectra of some families of graphs}

In this last section, we show the application of our method to obtain the local spectra and the complete spectrum of different well-known families of graphs.

\subsection{Walk-regular graphs}
Let $\G$ be a graph with spectrum as above.
If the number of closed walks of
length ${\ell}$ rooted at vertex $u$, that is,
$a_{uu}^{({\ell})}=\sum_{i=0}^d m_u(\theta_i)\theta_i^{\ell}$ only
depends on ${\ell}$, for each $\ell \ge 0$, then $\G$ is called
{\em walk-regular} (a concept introduced by Godsil and
McKay~\cite{gmk}). In this case, we write
$a_{uu}^{({\ell})}=a^{({\ell})}$. Note that, since
$a_{uu}^{(2)}=\delta(u)$, the degree of vertex $u$, a walk-regular
graph is necessarily regular.
Moreover, we say that $\G$ is {\em
spectrum-regular} if, for any $i=0,1,\ldots, d$, the $u$-local
multiplicity of $\theta_i$ does not depend on the vertex $u$.
By (\ref{crossed-mul->num-walks})
and the subsequent comment, it follows that
spectrum-regularity and walk-regularity are equivalent
concepts. Equation (\ref{crossed-mul->num-walks}) also shows that the existence of the constants
$a^{(0)},a^{(1)},\ldots,a^{(d)}$ suffices to assure
walk-regularity. It is well known that any distance-regular
graph, as well as any vertex-transitive graph, is walk-regular,
but the converse is not true.

\begin{proposition}
Let $\G$ be a walk-regular graph with $n$ vertices, having a regular partition  $\pi=(V_1,\ldots,V_m)$ with $V_1=\{u\}$ and quotient matrix $\B$.
Then, the spectrum of $\G$ is
$$
\spec \G = \{\theta_0^{m(\theta_0)},\theta_1^{m(\theta_1)},\dots,
\theta_d^{m(\theta_d)}\},
$$
where, for every $i=0,\ldots,d$, $\theta_i$ also is an eigenvalue of $\B$, with multiplicity
\begin{equation}
\textstyle
m(\theta_i)=n(\overline{\E}_i)_{11}=
\frac{n}{\prod_{\stackrel{j=0}{j\neq i}}^d
(\theta_i-\theta_j)}\left(\prod_{\stackrel{j=0}{j\neq i}}^d
(\B-\theta_j\I)\right)_{11}. \label{mult-walk-reg}
\end{equation}
\end{proposition}

\begin{proof}
Since $\sum_{u\in V}m_{uu}(\theta_i)=m(\theta_i)$,  the (standard)
multiplicity $m(\theta_i)$ `splits' equitably among the $n$ vertices, giving
$m_u(\theta_i)=m(\theta_i)/n$. Therefore, the different eigenvalues of $\B$ coincide with those of $\B$, and Theorem \ref{main-theo} yields \eqref{mult-walk-reg}.
\end{proof}

Let us see an example.
\begin{example}
Consider the walk-regular graph $\G$, that is not distance-regular, given by Godsil \cite{g93}. This graph and its quotient $\pi(\Gamma)$ are represented in Figure \ref{fig:godsil+quocient}.
The spectrum of $\pi(\Gamma)$ is $\spec \pi(\Gamma)=\{\theta_0^{m(\tau_0)},\theta_1^{m(\tau_1)},\theta_2^{m(\tau_2)},\theta_3^{m(\tau_3)}\}=\{4^1,2^2,0^1,-2^3,\}$. Since $\G$ is walk-regular, the spectrum of its quotient from a regular partition has all the different eigenvalues of the spectrum of $\G$. Now we compute the multiplicities $m(\theta_i)$, for $i=0,1,2,3$.
\begingroup
\everymath{\scriptstyle}
\small
\begin{eqnarray*}
&& m (\theta_0) = \frac {12} {(\theta_0- \theta_1) (\theta_0- \theta_2) (\theta_0- \theta_3)} \left((\B- \theta_1 \I) (\B- \theta_2 \I) (\B- \theta_3 \I)\right)_{11} = 1, \\
&& m (\theta_1) = \frac {12} {(\theta_1- \theta_0) (\theta_1- \theta_2) (\theta_1- \theta_3)} \left((\B- \theta_0 \I) (\B- \theta_2 \I) (\B- \theta_3 \I)\right)_{11} = 3, \\
&& m (\theta_2) = \frac {12} {(\theta_2- \theta_0) (\theta_2- \theta_1) (\theta_2- \theta_3)} \left((\B- \theta_0 \I) (\B- \theta_1 \I) (\B- \theta_3 \I)\right)_{11} = 3, \\
&& m (\theta_3) = \frac {12} {(\theta_3- \theta_0) (\theta_3- \theta_1) (\theta_3- \theta_2)} \left((\B- \theta_0 \I) (\B- \theta_1 \I) (\B- \theta_2 \I)\right)_{11} = 5.
\end{eqnarray*}
\endgroup
This gives that the spectrum of $\G$ is $\spec \G=\{4^1,2^3,0^3,-2^5,\}$, as it is known to be. Note that, in this example, we need to `hang' the graph $\G$ from only one of its vertices.
\end{example}

\begin{figure}[t]
    \begin{center}
        \includegraphics[width=18cm]{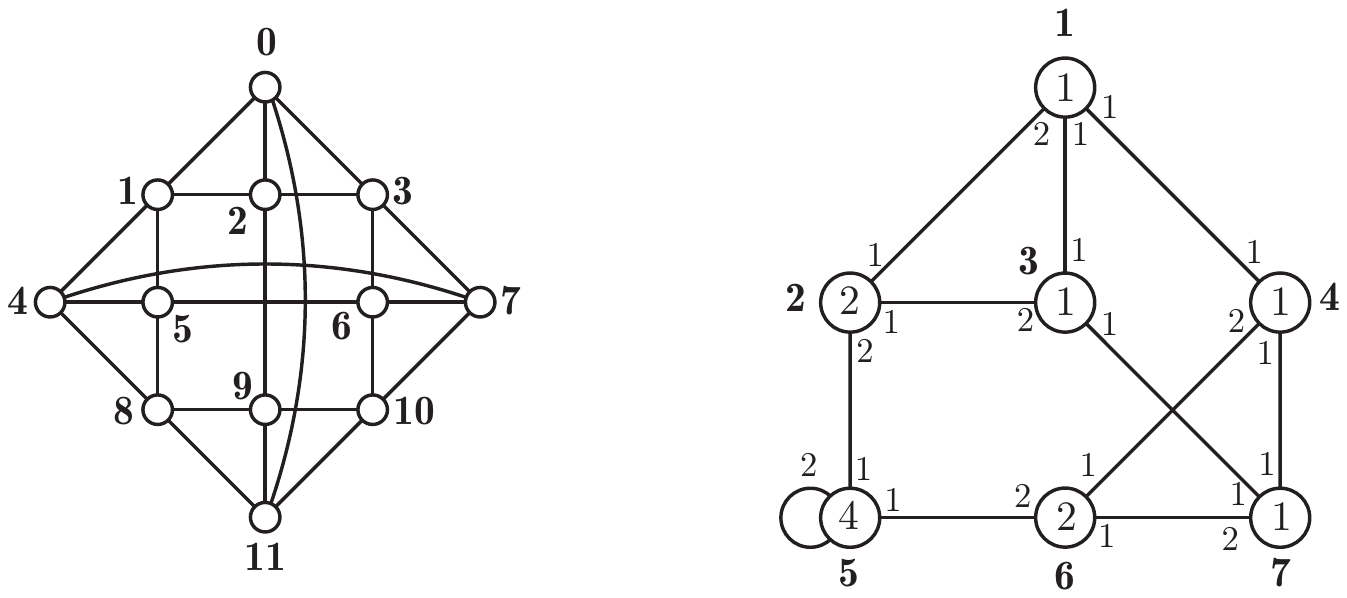}
    \end{center}
    \vskip-20cm
	\caption{Left: The walk-regular, but  not distance-regular, graph $\G$ given by Godsil \cite{g93}. Right: The quotient graph $\pi(\G)$ of $\G$. In boldface, there are the numberings of the vertices.}
	\label{fig:godsil+quocient}
\end{figure}

\subsection{Distance-regular graphs}
In particular, when $\G$ is distance-regular, the distance-partition with respect to any vertex is regular with the same quotient matrix $\B$ (see, for instance, Biggs \cite{biggs} or Fiol \cite{fiol02}). Moreover, since $\B$ is tridiagonal, all its eigenvalues are simple and \eqref{mult-walk-reg}, together with \eqref{Biggs-general}, leads to the known formula (see Biggs \cite{biggs})
\begin{equation}
m(\theta_i)=n(\overline{\E}_i)_{11}=\frac{n}{\langle\u_i,\v_i \rangle}, \qquad j=1,\ldots,m,
\end{equation}
where the eigenvectors $\u_i$ and $\v_i$ have been chosen to have the first entry $1$.

\subsection{Distance-biregular graphs}
Distance-biregular graphs are defined in a similar way as distance-regular graphs. They are connected bipartite graphs in which each of the two classes of vertices has its own intersection array. It was proved by Godsil and Shawe-Taylor \cite{gs87} that all vertices in the same bipartition class have the same intersection array. Delorme \cite{d94} gave the basic properties and some new examples of distance-biregular graphs.

\begin{figure}[t]
    \begin{center}
        \includegraphics[width=24cm]{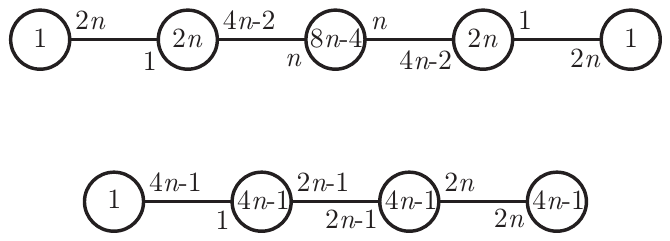}
    \end{center}
    \vskip-30.5cm
	\caption{The quotient graphs of the Hadamard distance-biregular graphs.}
	\label{fig:Hadamard-quocient}
\end{figure}

Let us give an example.
\begin{example}
A Hadamard matrix $H$ (with entries $\pm 1$ and mutually orthogonal rows) with size $4n$ gives a bipartite distance-biregular graph $Hb(n)$ on $12n-2$ vertices (see, for instance, Wallis \cite[p. 426]{w74}). Its stable sets $V_1$ and $V_2$ have $4n$ and $8n-2$ vertices, respectively. For instance, $Hb(1)$ is the subdivided complete graph $K_4$. The quotient graphs corresponding to the regular distance-partitions with respect to vertices in $V_1$ and $V_2$ are shown in Figure \ref{fig:Hadamard-quocient}. Thus, the respective quotient matrices are
$$
\B_1=\left(
\begin{array}{ccccc}
0 & 2n & 0 & 0 & 0\\
1 & 0 & 4n-2 & 0 & 0\\
0 & n & 0 & n & 0\\
0 & 0 & 4n-2 & 0 & 1\\
0 & 0 & 0 & 2n & 0
\end{array}
\right),\qquad
\B_2=\left(
\begin{array}{cccc}
0 & 4n-1 & 0 & 0\\
1 & 0 & 2n-1 & 0\\
0 & 2n-1 & 0 & 2n\\
0 & 0 & 2n & 0
\end{array}
\right),
$$
with (simple) eigenvalues
$
\ev \B_1 =\{\sqrt{8n^2-1}, \sqrt{2n}, 0, -\sqrt{2n}, -\sqrt{8n^2-1}\}
$
and
$
\ev \B_2 =\ev \B_1\setminus \{0\}.
$
Then, according to Theorem \ref{main-theo}, we can compute all the local (crossed) multiplicities of vertices in each stable set from the idempotents of $\B_1$ and $\B_2$. The results obtained are shown in Table \ref{table1} (for $u\in V_1$) and Table \ref{table2} (for $u\in V_2$), where the last row in both tables corresponds to the sums in \eqref{crossed-mul->num-walks} for $\ell=0$ (or, for the case $u=v$, to \eqref{sum-locmul1}). Moreover, from the columns of local multiplicities ($\dist(u,v)=0$), we can find the (global) multiplicities by using \eqref{sum-locmul2}, which in our case becomes
$$
m(\theta_i)=\sum_{u\in V_1}m_u(\theta_i)+\sum_{v\in V_2}m_v(\theta_i)= (8n-2)\cdot m_u(\theta_i)+4n\cdot m_v(\theta_i).
$$
Then, the complete spectrum of the Hadamard distance-biregular graph turns out to be
$$
\spec Hb(n)=\{\sqrt{8n^2-2n}^{1}, \sqrt{2n}^{4n-1},
0^{4n-2}, -\sqrt{2n}^{4n-1},-\sqrt{8n^2-2n}^{1}\}.
$$
\end{example}

\newpage

\begin{table}[h!]
\begin{center}
\begin{tabular}{|c|ccccc|}
\hline\\[-.4cm]
$\dist(u,v)$, $u\in V_1$ &  0& 1                  & 2  &3                  &4 \\[.2cm]
\hline
\hline\\[-.4cm]
$m_{uv}(\theta_0)$            & $\frac{1}{16n-4}$   & $\frac{\sqrt{2}}{8\sqrt{4n^2-n}}$  & $\frac{1}{16n-4}$  & $\frac{\sqrt{2}}{8\sqrt{4n^2-n}}$  & $\frac{1}{16n-4}$  \\[.2cm]
$m_{uv}(\theta_1)$            & $\frac{1}{4}$       & $\frac{\sqrt{2}}{8\sqrt{n}}$       & $0$                & $-\frac{\sqrt{2}}{8\sqrt{n}}$      & $-\frac{1}{4}$  \\[.2cm]
$m_{uv}(\theta_2)$            & $\frac{2n-1}{4n-1}$ & $0$                                & $-\frac{1}{8n-2}$  & $0$                                & $\frac{2n-1}{4n-1}$  \\[.2cm]
$m_{uv}(\theta_3)$            & $\frac{1}{4}$       & $-\frac{\sqrt{2}}{8\sqrt{n}}$      & $0$                & $\frac{\sqrt{2}}{8\sqrt{n}}$       & $-\frac{1}{4}$  \\[.2cm]
$m_{uv}(\theta_4)$            & $\frac{1}{16n-4}$   & $-\frac{\sqrt{2}}{8\sqrt{4n^2-n}}$ & $\frac{1}{16n-4}$  & $-\frac{\sqrt{2}}{8\sqrt{4n^2-n}}$ & $\frac{1}{16n-4}$  \\[.2cm]
\hline\\[-.4cm]
$\sum_{i=0}^4 m_{uv}(\theta_i)$          & $1$                 & $0$                                & $0$                & $0$                                & $0$ \\[.2cm]
\hline
\end{tabular}
\end{center}
\caption{Local multiplicities, from a vertex $u\in V_1$, of the Hadamard distance-biregular graph $Hb(n)$.}
\label{table1}
\end{table}
\begin{table}[h!]
\begin{center}
\begin{tabular}{|c|cccc|}
\hline\\[-.4cm]
$\dist(u,v)$, $u\in V_2$&  0& 1                  & 2  &3                   \\[.2cm]
\hline
\hline\\[-.4cm]
$m_{uv}(\theta_0)$            & $\frac{1}{8n}$      & $\frac{\sqrt{2}}{8\sqrt{4n^2-n}}$                 & $\frac{1}{8n}$                     & $\frac{\sqrt{2}}{8\sqrt{4n^2-n}}$ \\[.2cm]
$m_{uv}(\theta_1)$            & $\frac{4n-1}{8n}$   & $\frac{\sqrt{2}}{8\sqrt{n}}$                     & $-\frac{1}{8n}$                    & $-\frac{\sqrt{2}}{8\sqrt{n}}$  \\[.2cm]
$m_{uv}(\theta_2)$           & $0$                      & $0$      & $0$                & $0$\\[.2cm]
$m_{uv}(\theta_3)$            & $\frac{4n-1}{8n}$   & $-\frac{\sqrt{2}}{8\sqrt{n}}$                    & $-\frac{1}{8n}$                    & $\frac{\sqrt{2}}{8\sqrt{n}}$  \\[.2cm]
$m_{uv}(\theta_4)$           & $\frac{1}{8n}$      & $-\frac{\sqrt{2}}{8\sqrt{4n^2-n}}$               & $\frac{1}{8n}$                     & $-\frac{\sqrt{2}}{8\sqrt{4n^2-n}}$\\[.2cm]

\hline\\[-.4cm]
$\sum_{i=0}^3 m_{uv}(\theta_i)$         & $1$                 & $0$                                               & $0$                                & $0$ \\[.2cm]
\hline
\end{tabular}
\end{center}
\caption{Local multiplicities, from a vertex $u\in V_2$, of the Hadamard distance-biregular graph $Hb(n)$.}
\label{table2}
\end{table}

%


\end{document}